\documentclass{amsart}

\usepackage{graphicx}
\usepackage{bm}
\usepackage{color}
\usepackage{natbib}
\usepackage{amssymb, amsmath, amsthm}
\usepackage{bm}
\theoremstyle{plain}
\usepackage{graphicx}
\usepackage[aboveskip=-5pt,position=bottom]{caption}
\usepackage{url}
\usepackage{hyperref}

\newtheorem{lemma}{Lemma}[section]

\newtheorem{theorem}[lemma]{Theorem}

\newtheorem{prop}[lemma]{Proposition}
\newtheorem{exam}[lemma]{\normalfont \scshape
 Example}

\newcommand{\R}{\mathbb{R}}
\newcommand{\N}{\mathbb{N}}
\newcommand{\norm}[1]{\left\Vert#1\right\Vert}
\newcommand{\abs}[1]{\left\vert#1\right\vert}
\newcommand{\set}[1]{\left\{#1\right\}}

\newcommand{\bfx}{\bm{x}}
\newcommand{\bfzero}{\bm{0}}

\newcommand{\bfone}{\bm{1}}

\newcommand{\bfu}{\bm{u}}

\newcommand{\bfX}{\bm{X}}
\newcommand{\bfY}{\bm{Y}}

\newcommand{\bfZ}{\bm{Z}}

\newcommand{\bfeta}{\bm{\eta}}

\allowdisplaybreaks[4]
\begin{document}
\title[The Space of $D$-Norms Revisited]{The Space of $D$-Norms Revisited}%
\author{Stefan Aulbach, Michael Falk and Maximilian Zott}
\address{University of W\"{u}rzburg, Institute of Mathematics,  Emil-Fischer-Str. 30, 97074 W\"{u}rzburg, Germany.}
\email{stefan.aulbach@uni-wuerzburg.de, michael.falk@uni-wuerzburg.de,\linebreak
maximilian.zott@uni-wuerzburg.de}


\dedicatory{The final publication is available at Springer via
\url{http://dx.doi.org/10.1007/s10687-014-0204-y}}

\begin{abstract}
The theory of $D$-norms is an offspring of multivariate extreme value
theory. We present recent results on $D$-norms, which are completely
determined by a certain random vector called generator. In the first part
it is shown that the space of $D$-norms is a complete separable metric
space, if equipped with the Wasserstein-metric in a suitable way. Secondly,
multiplying a generator with a doubly stochastic matrix yields another
generator. An iteration of this multiplication provides a sequence of
$D$-norms and we compute its limit. Finally, we consider a parametric
family of $D$-norms, where we assume that the generator follows a symmetric
Dirichlet distribution. This family covers the whole range between complete
dependence and independence.
\end{abstract}
\maketitle

\section{Introduction}\label{sec:Introduction}
A norm $\norm\cdot_D$ on $\R^d$ is a $D$-norm, if there exists a random
variable (rv) $\bfZ=(Z_1,\dots,Z_d)$ with $Z_i\ge 0$, $E(Z_i)=1$, $1\le i\le
d$, such that
\begin{equation*}
\norm{\bfx}_D=E\left(\max_{1\le i\le d}\left(\abs{x_i}Z_i\right)\right)
=E\left(\norm{\bfx\bfZ}_\infty\right),
\end{equation*}
$\bfx=(x_1,\dots,x_d)\in\R^d$. In this case $\bfZ$ is called \emph{generator}
of $\norm\cdot_D$. By $\norm{\bfx}_\infty=\max_{1\le i\le d}\abs{x_i}$ we
denote the usual sup-norm on $\R^d$; all operations on vectors such as
$\bfx\bfZ=(x_iZ_i)_{i=1}^d$ are meant componentwise.

Examples of $D$-norms are
\begin{enumerate}
  \item the sup-norm $\norm{\bfx}_\infty = \max_{1\le i\le d}\abs{x_i}$,
      which is generated by $\bfZ=(1,\dots,1)$.
  \item the $L_1$-norm $\norm{\bfx}_1=\sum_{i=1}^d\abs{x_i}$, generated by a
      random permutation of $(d,0,\dots,0)\in\R^d$ with equal probability
      $1/d$.
  \item the usual logistic-norm $
      \norm{\bfx}_\lambda=\left(\sum_{i=1}^d\abs{x_i}^\lambda\right)^{1/\lambda}$,
      $1<\lambda<\infty$. An explicit generator was only quite recently
      found: Let $X_1,\dots,X_d$ be independent and identically
      Fr\'{e}chet-distributed rv, i.e., $P(X_i\le x)=$ $\exp(-x^{-\lambda})$,
      $x>0$, $\lambda>1$. Then $\bfZ=(Z_1,\dots,Z_d)$ with
      \begin{equation*}\label{eq:frechet_generator}
      Z_i:=\frac{X_i}{\Gamma(1-p^{-1})},\quad i=1,\dots,d,
      \end{equation*}
      generates $\norm\cdot_{\lambda}$ where $\Gamma$ denotes the gamma
      function.
\end{enumerate}

The theory of $D$-norms is an offspring of multivariate extreme value theory:
A distribution function (df) $G$ on $\R^d$ is a \emph{standard max-stable}
(sms) or \emph{standard extreme value df} if
\begin{gather*}
G(\bfx)=G^n\left(\frac{\bfx}n\right),\qquad \bfx\in\R^d,\,n\in\N,\\
\intertext{and for $1\le i\le d$}
G_i(\bfx):=G(0,\dots,0,x_i,0,\dots,0)=\exp(x_i),\qquad \bfx\le\bfzero\in\R^d.
\end{gather*}
The following characterization of a sms df in terms of a $D$-norm is a
consequence of the results by \citet{pick75}, \citet{dehar77} and
\citet{vatan85}.

\begin{theorem}[Pickands, de Haan-Resnick, Vatan]
A df $G$ on $\R^d$ is a sms df iff there exists a $D$-norm $\norm\cdot_D$ on
$\R^d$ such that
\[
G(\bfx)=\exp\left(-\norm{\bfx}_D\right),\qquad \bfx\le\bfzero\in\R^d.
\]
\end{theorem}

The generator $\bfZ$ of a $D$-norm $\norm\cdot_D$ is in general not uniquely
determined, even its distribution is not, cf.
\eqref{eqn:representation_of_Dirichlet_D-norm} in Section
\ref{sec:dirichlet_distribution}. The sup-norm $\norm\cdot_\infty$, for
example, can be generated by every rv $\bfZ=(Z,\dots,Z)$ with constant entry
$Z$ which is a positive rv with expectation 1.

The particular value
\[
\norm{\bfone}_D=E\left(\max_{1\le i\le d}Z_i\right)
\]
of a $D$-norm on $\R^d $ with generator $\bfZ$ is the \emph{generator
constant} or \emph{extremal coefficient}, cf. \citet{smith90}, where $\bfone
= (1,\dots,1)$. While a generator is in general not uniquely determined by
the $D$-norm, the generator constant obviously is. It is a measure of
dependence between the margins of the multivariate sms df
$G(\bfx):=\exp\left(-\norm{\bfx}_D\right)$, $\bfx\le\bfzero\in\R^d$, see
\citet[Section 4.4]{fahure10}. We have by \citeauthor{taka88}'s
(\citeyear{taka88}) theorem
\[
\norm\cdot_D=\norm\cdot_1\iff \norm{\bfone}_D=d,
\]
which is the case of independence of the margins of $G$, and
\[
\norm\cdot_D=\norm\cdot_\infty\iff \norm{\bfone}_D=1,
\]
which is the case of complete dependence of the margins. Note that
\begin{equation} \label{eqn:D-norm_bounds}
\norm\cdot_\infty\le\norm\cdot_D\le\norm\cdot_1
\end{equation}
for any $D$-norm, with the lower and the upper bound being $D$-norms
themselves.

A rv $\bfeta$ that follows the sms df
$G(\bfx)=\exp\left(-\norm{\bfx}_D\right)$, $\bfx\le\bfzero\in\R^d$, can be
generated in the following way. Consider a Poisson point process on
$[0,\infty)$ with mean measure $r^{-2}dr$. Let $V_i$, $i\in\N$, be a
realization of this point process. Consider independent copies
$\bfZ^{(1)},\bfZ^{(2)},\dots$ of a generator $\bfZ$ of the $D$-norm
$\norm{\cdot}_D$, which are also independent of the Poisson process. Then we have
\[
\bfeta=_D-\frac 1{\sup_{i\in\N} V_i\bfZ^{(i)}},
\]
which is a consequence of \citet[Lemma 9.4.7]{dehaf06} and elementary
computations.

Let $\norm\cdot_{D_1}$, $\norm\cdot_{D_2}$ be two $D$-norms on $\R^d$ with
generators $\bfZ^{(1)}$, $\bfZ^{(2)}$. Suppose that these generators are
independent. Then the  product $\bfZ:=\bfZ^{(1)}\bfZ^{(2)}$, taken
componentwise, defines the generator of a $D$-norm $\norm\cdot_{D_1\times
D_2}$, say.  This entails the definition of a multiplication type operation
on the set of $D$-norms; note that this product $D$-norm does not depend on
the special choice of generators. A $D$-norm $\norm\cdot_D$ is called
\emph{idempotent}, if $\norm\cdot_{D\times D}=\norm\cdot_D$. The sup-norm
$\norm\cdot_\infty$ and the $L_1$-norm $\norm\cdot_1$ are idempotent
$D$-norms. Iterating the multiplication provides a track of $D$-norms, whose
limit exists and is again a $D$-norm. If this iteration is repeatedly done on
the same $D$-norm, then the limit of the track is idempotent, see
\citet{falk13}, where also the set of idempotent $D$-norms is characterized.

In Section \ref{sec:metrization_of_D-norm-space} of the present paper we
define a metric on the space of $D$-norms such that it becomes a complete
metric space. Convergence of $D$-norms is then equivalent with weak
convergence of the corresponding generators. Multiplying a generator with a
bistochastic or doubly stochastic matrix generates a new generator and, thus,
another $D$-norm. Iterating the multiplication leads to a sequence of
$D$-norms, whose limit is established in Section
\ref{sec:bistochastic_matrices}. A particularly interesting parametric model
for generators is provided by the symmetric Dirichlet-distributions. In
Section \ref{sec:dirichlet_distribution} we investigate this parametric
family in detail.

\section{Metrization of the Space of $D$-Norms} \label{sec:metrization_of_D-norm-space}

Denote by $\mathcal Z_{\norm\cdot_D}$ the set of all generators of a given
$D$-norm $\norm\cdot_D$ on $\R^d$. The proof of the de Haan-Resnick
representation of a max-stable multivariate extreme value df as in
\citet[Section 4.2]{fahure10} implies the following result.

\begin{lemma}
Each set $\mathcal Z_{\norm\cdot_D}$ contains a generator $\bfZ$ with the
additional property $\norm{\bfZ}_1=d$. The distribution of this $\bfZ$ is
uniquely determined.
\end{lemma}

Let $\mathbb{P}$ be the set of all probability measures on $S_d:=\set{\bfx\ge
\bfzero\in\R^d:\,\norm{\bfx}_1=d}$. We, thus, can identify the set $\mathbb
D$ of $D$-norms on $\R^d$ with the subset $\mathbb P_D$ of those probability
distributions $P\in\mathbb P$ which satisfy the additional condition
$\int_{S_d}x_i\,P(d\bfx)=1$, $i=1,\dots,d$.

Denote by $d_W(P,Q)$ the \emph{Wasserstein metric} between two probability
distributions on $S_d$, i.e.,
\begin{equation*}
  d_W(P,Q)
  :=\inf\set{E\left(\norm{\bfX-\bfY}_1\right):\, \bfX\mathrm{\ has\ distribution\ }P,\,\bfY \mathrm{\ has\ distribution\ }Q}.
\end{equation*}
As $S_d$, equipped with an arbitrary norm $\norm\cdot$, is a complete
separable space, the metric space $\left(\mathbb P,d_W\right)$ is complete
and separable as well; see, e.g., \citet{bolley08}.

\begin{lemma}\label{lem:completeness_of_P_D}
The subspace $\left(\mathbb P_D,d_W\right)$ of $\left(\mathbb P,d_W\right)$
is also separable and complete.
\end{lemma}

\begin{proof}
Let $P_n$, $n\in\N$, be a sequence in $\mathbb P_D$, which converges with
respect to $d_W$ to $P\in\mathbb P$. We show that $P\in\mathbb P_D$. Let the
rv $\bfX$ have distribution $P$ and let $\bfX^{(n)}$ have distribution $P_n$,
$n\in\N$. Then we have
\begin{align*}
\sum_{i=1}^d \abs{\int_{S_d} x_i\,P(d\bfx) - 1}&= \sum_{i=1}^d \abs{\int_{S_d} x_i\,P(d\bfx) - \int_{S_d} x_i\,P_n(d\bfx)}\\
&= \sum_{i=1}^d \abs{ E\left(X_i-X_i^{(n)}\right)}\\
&\le E\left( \sum_{i=1}^d \abs{ X_i-X_i^{(n)}}\right)\\
&= E\left(\norm{\bfX-\bfX^{(n)}}_1\right),\qquad n\in\N.
\end{align*}
As a consequence we obtain
\[
\sum_{i=1}^d \abs{\int_{S_d} x_i\,P(d\bfx) - 1}\le d_W(P,P_n)\to_{n\to\infty} 0,
\]
and, thus, $P\in\mathbb P_D$. The separability of $\mathbb P_D$ can be seen
as follows. Let $\mathcal P$ be a countable and dense subset of $\mathbb P$.
Identify each distribution $P$ in $\mathcal P$ with a rv $\bfY$ on $S_d$ that
follows this distribution $P$. Put $\bfZ=\bfY/E(\bfY)$, where we can assume
that each component of $\bfY$ has positive expectation. This yields a
countable subset of $\mathbb P_D$, which is dense.
\end{proof}

We can now define the distance between two $D$-norms $\norm\cdot_{D_1}$,
$\norm\cdot_{D_2}$ on $\R^d$ by
\begin{align*}
&d_W\left(\norm\cdot_{D_1}, \norm\cdot_{D_2}\right)\\
&:=\inf\set{ E\left(\norm{\bfZ^{(1)}-\bfZ^{(2)}}_1\right):\,\bfZ^{(i)}\mathrm{\ generates\ }\norm\cdot_{D_i},\,\norm{\bfZ^{(i)}}_1=d,\,i=1,2}.
\end{align*}
The space $\mathbb D$ of $D$-norms on $\R^d$, equipped with the distance
$d_W$, is by Lemma \ref{lem:completeness_of_P_D} a complete and separable
metric space.

For the rest of this section we restrict ourselves to generators $\bfZ$ of
$D$-norms on $\R^d$ that satisfy $\norm{\bfZ}_1=d$.

\begin{lemma}
Let $\norm\cdot_{D_n}$, $n\in\N\cup\set{0}$, be a sequence of $D$-norms on
$\R^d$ with corresponding generators $\bfZ^{(n)}$, $n\in\N\cup\set{0}$. Then
we have the equivalence
\[
d_W\left(\norm\cdot_{D_n},\norm\cdot_{D_0}\right)\to_{n\to\infty} 0 \iff \bfZ^{(n)}\to_D\bfZ^{(0)},
\]
where $\to_D$ denotes ordinary convergence in distribution.
\end{lemma}

\begin{proof}
Convergence of probability measures $P_n$ to $P_0$ with respect to the
Wasserstein-metric is equivalent with weak convergence together with
convergence of the moments
\[
\int_{S_d}\norm{\bfx}_1\,P_n(d\bfx) \to_{n\to\infty} \int_{S_d}\norm{\bfx}_1\,P_0(d\bfx),
\]
see, e.g., \citet{villani09}. But as we have for each probability measure
$P\in\mathbb P_D$
\[
\int_{S_d}\norm{\bfx}_1\,P(d\bfx)=\int_{S_d}d\,P(d\bfx)=d,
\]
convergence of the moments is automatically satisfied.
\end{proof}

\begin{lemma}
We have for arbitrary $D$-norms $\norm\cdot_{D_1}$, $\norm\cdot_{D_2}$ on
$\R^d$ the bound
\begin{equation*}
\norm{\bfx}_{D_1}\le \norm{\bfx}_{D_2} + \norm{\bfx}_\infty d_W\left(\norm\cdot_{D_1}, \norm\cdot_{D_2}\right)
\end{equation*}
and, thus,
\[
\sup_{\bfx\in\R^d,\norm{\bfx}_\infty\le r}\abs{\norm{\bfx}_{D_1} -
\norm{\bfx}_{D_2}} \le r\, d_W\left(\norm\cdot_{D_1},
\norm\cdot_{D_2}\right),\qquad r\ge 0.
\]
\end{lemma}

\begin{proof}
Let $\bfZ^{(i)}$ be a generator of $\norm\cdot_{D_i}$, $i=1,2$. We have
\begin{align*}
\norm{\bfx}_{D_1}&= E\left(\max_{1\le i\le d}\left(\abs{x_i} Z_i^{(1)}\right)\right)\\
&= E\left(\max_{1\le i\le d}\left(\abs{x_i}\left(Z_i^{(2)}+ Z_i^{(1)}-Z_i^{(2)}\right)\right)\right)\\
&\le E\left(\max_{1\le i\le d}\left(\abs{x_i} Z_i^{(2)}\right)\right) + \norm{\bfx}_\infty E\left(\max_{1\le i\le d}\abs{Z_i^{(1)}-Z_i^{(2)}}\right),
\end{align*}
which implies the assertion.
\end{proof}

\section{Doubly Stochastic Matrices}\label{sec:bistochastic_matrices}

Denote by $\mathbb M$ the set of all doubly stochastic (or bistochastic)
$d\times d$-matrices. Let $\bfZ$ be the generator of a $D$-norm
$\norm\cdot_D$ on $\R^d$ with the additional property $\norm{\bfZ}_1=d$. If
$\bfZ$ is interpreted as a column vector then
\[
\bfZ_M:= M\bfZ
\]
is for each $M\in\mathbb M$ the generator of a $D$-norm as well. By the fact
that $M$ is doubly stochastic, we also have $\norm{\bfZ_M}_1=d$.

Let, for instance, $\bfZ$ be a random permutation of the vector
$(d,0,\dots,0)^\intercal\in\R^d$ with equal probability $1/d$. The
corresponding $D$-norm is $\norm\cdot_1$, which is an upper bound for each
$D$-norm. Let $M_0$ be the $d\times d$-matrix with constant entry $1/d$. Then
we obtain
\[
\bfZ_{M_0}= M_0\bfZ=(1,\dots,1)^\intercal,
\]
which is the generator of the $D$-norm $\norm\cdot_\infty$. This $D$-norm is
a lower bound for each $D$-norm.  This example shows the influence  that the
multiplication of a generator with a doubly stochastic matrix can have. Note
that actually $M_0\bfZ=(1,\dots,1)^\intercal$ for each generator $\bfZ$
satisfying $\norm{\bfZ}_1=d$.

By identifying a generator $\bfZ$ with its corresponding $D$-norm
$\norm\cdot_{D(\bfZ)}$, say, we define the function $f:\mathbb
M\times\mathbb D\to\mathbb D$ by
\[
f\left(M,\norm\cdot_{D(\bfZ)}\right):=\norm\cdot_{D(M\bfZ)};
\]
recall that the distribution of the generator $\bfZ$ of a $D$-norm is
uniquely determined under the additional condition $\norm{\bfZ}_1=d$.

\begin{lemma}\label{lem:distance_between_norms}
If we equip $\mathbb M$ with  the metric
$\norm{M_1-M_2}_1=\sum_{i,j=1}^d\abs{m_{ij}^{(1)}- m_{ij}^{(2)}}$,
$M_1,M_2\in\mathbb M$, and the space $\mathbb D$ of all $D$-norms on $\R^d$
with the Wasserstein metric $d_W$, then the function $f$ is continuous,
precisely,
\begin{align*}
&d_W\left(f\left(M_1,\norm\cdot_{D(\bfZ^{(1)})}\right), f\left(M_2,\norm\cdot_{D(\bfZ^{(2)})}\right)\right)\\
&\le \norm{M_1-M_2}_1 + d\, d_W\left(\norm\cdot_{D(\bfZ^{(1)})}, \norm\cdot_{D(\bfZ^{(2)})}\right).
\end{align*}
\end{lemma}

\begin{proof}
The triangular inequality implies
\begin{align*}
&d_W\left(f\left(M_1,\norm\cdot_{D(\bfZ^{(1)})}\right), f\left(M_2,\norm\cdot_{D(\bfZ^{(2)})}\right)\right)\\
&\le d_W\left(f\left(M_1,\norm\cdot_{D(\bfZ^{(1)})}\right), f\left(M_2,\norm\cdot_{D(\bfZ^{(1)})}\right)\right)\\
 &\hspace*{1cm}+ d_W\left(f\left(M_2,\norm\cdot_{D(\bfZ^{(1)})}\right), f\left(M_2,\norm\cdot_{D(\bfZ^{(2)})}\right)\right)\\
&\le E\left(\norm{(M_1-M_2)\bfZ^{(1)}}_1\right) + E\left(\norm{M_2\left(\bfZ^{(1)}-\bfZ^{(2)}\right)}_1\right)\\
&\le \norm{M_1-M_2}_1 + dE\left( \norm{\bfZ^{(1)}-\bfZ^{(2)}}_1\right),
\end{align*}
which yields the assertion.
\end{proof}

Let $\bfZ$ be a random permutation of the  vector $(d,0,\dots,0)\in\R^d$ with
equal probability $1/d$ and set $M_0=(1/d)\in\R^{d\times d}$. Then we obtain
from Lemma \ref{lem:distance_between_norms} the bound
\begin{equation*}
d_W\left(\norm\cdot_\infty,\norm\cdot_1\right) = d_W\left(f\left(M_0,\norm\cdot_1\right), f\left(I_d,\norm\cdot_1\right)\right)
\le \norm{M_0-I_d}_1=2(d-1),
\end{equation*}
where $I_d$ is the $d\times d$ unit matrix. Note that this bound is sharp by
the fact that the distribution of a generator $\bfZ$ with $\norm{\bfZ}_1=d$
is uniquely determined and, thus, we compute
$d_W\left(\norm\cdot_\infty,\norm\cdot_1\right)=2(d-1)$.

The idea suggests itself to iterate the multiplication of a generator with a
matrix and to consider
\[
\bfZ^{(n)}:= M^n\bfZ,\qquad n\in\N,
\]
where $M^n$ denotes the ordinary $n$-times matrix product. The question,
whether the sequence $\bfZ^{(n)}$, $n\in\N$, converges, can be answered by
fundamental results from the theory of Markov chains. In particular we obtain
the following result, which shows that the sequence of $D$-norms
$\bigl(\norm\cdot_{D(\bfZ^{(n)})}\bigr)_{n\in\N}$ converges to $\norm\cdot_1$
under mild conditions on the matrix $M$.

\begin{prop}
Suppose that each entry $M^n(i,j)$ of the matrix $M^n$ is positive if $n$ is
large. Then we obtain for an arbitrary generator $\bfZ$
\[
\bfZ^{(n)}\to_{n\to\infty} (1,\dots,1)^\intercal \in\R^d.
\]
\end{prop}

The condition $M^n(i,j)>0$ for each $i,j\in\set{1,\dots,d}$ cannot be dropped
in the preceding result; just set $M=I_d$, the unit matrix, or let $M$ be any
bistochastic matrix which has only the entries zero and one.

\begin{proof}
The matrix $M=(m(i,j))_{1\le i,j\le d}$ can be viewed as a matrix of
transition probabilities $p(j\mid i)$ from the state $i$ to the state $j$,
where $i,j\in\set{1,\dots,d}$, and, thus, the transition matrix $M$ defines a
time-homogenous  Markov chain on the state space $\set{1,\dots,d}$. The
condition that each entry of $M^n$ is positive for large $n$ is equivalent
with the condition that $M$ is aperiodic and irreducible. It is well-known
from the theory of Markov chains that in this case
\[
M^n(i,j)\to_{n\to\infty}\mu(j), \qquad i,j\in\set{1,\dots,d},
\]
where the (row) vector $\mu$ is the uniquely determined stationary
distribution on $\set{1,\dots,d}$, i.e., $\mu M=\mu$.  As $M$ is
bistochastic, we obtain $\mu(j)=1/d$, $j=1,\dots,d$, which completes the
proof.
\end{proof}

\section{The $D$-Norm Generated From a Symmetric Dirichlet Distribution}\label{sec:dirichlet_distribution}

Let in what follows $V_1,\dots,V_d$, $d\ge 2$, be independent and identically
gamma distributed rv with density
$\gamma_\alpha(x):=x^{\alpha-1}\exp(-x)/\Gamma(\alpha)$, $x>0$, $\alpha>0$.
Then the rv $\tilde\bfZ\in\R^d$ with components
\[
\tilde Z_i:=\frac{V_i}{V_1+\dots+V_d},\qquad i=1,\dots,d,
\]
follows a symmetric \emph{Dirichlet distribution} $\mathrm{Dir}(\alpha)$ on
the closed simplex $\tilde S_d=\set{\bfu\ge
\bfzero\in\R^d:\,\sum_{i=1}^du_i=1}$, see \citet[Theorem 2.1]{ngtita2011}. By
equation (2.6) in this reference we have $E(\tilde Z_i)=1/d$ and, thus,
\begin{equation}\label{eqn:definition_generator_Dirichlet_distribution}
\bfZ:=d \tilde \bfZ
\end{equation}
is a generator of a $D$-norm $\norm\cdot_{D(\alpha)}$ on $\R^d$, which we
call the \emph{Dirichlet D-norm} with parameter $\alpha$. We have in
particular $\norm{\bfZ}_1=d$.

Note that $\gamma_1(x)=\exp(-x)$, $x>0$, is the density of the standard
exponential distribution, in which case
\[
\tilde\bfZ=_D\left(U_{i:d-1}-U_{i-1:d-1}\right)_{i=1}^d,
\]
where $U_{1:d-1}\le U_{2:d-1}\le\cdots\le U_{d-1:d-1}$ are the order
statistics pertaining to $d-1$ independent and on $(0,1)$ uniformly
distributed rv, $U_{0:d-1}:=0$, $U_{d:d-1}:=1$, see \citet[Theorem
1.6.7]{reiss89}. The distribution of the rv $\tilde\bfZ$ with $\alpha=1$ is,
therefore, that of the vector of uniform spacings.

It is well-known that for a general $\alpha>0$ the rv $\left(V_i/\sum_{j=1}^d
V_j\right)_{i=1}^d$ and the sum $\sum_{j=1}^d V_j$ are independent, see,
e.g., the proof of Theorem 2.1 in \citet{ngtita2011}. As
$E(V_1+\dots+V_d)=d\alpha$, we obtain for $\bfx = (x_1,\dots,x_d)\in\R^d$
\begin{align}
  \norm\bfx_{D(\alpha)}
  &=E\left(\max_{1\le i\le d}\left(\abs{x_i} Z_i\right)\right) \nonumber\\
  &= d E\left( \frac{\max_{1\le i\le d}\left(\abs{x_i} V_i\right)} {V_1+\dots+V_d}\right)\nonumber\\
  &=\frac 1\alpha E(V_1+\dots+V_d) E\left( \frac{\max_{1\le i\le d}\left(\abs{x_i} V_i\right)} {V_1+\dots+V_d}\right)\nonumber\\
  &= \frac 1\alpha E\left( \max_{1\le i\le d}\left(\abs{x_i} V_i\right) \right).
  \label{eqn:representation_of_Dirichlet_D-norm}
\end{align}
Note that the independence of $V_i/(V_1+\dots+V_d)$ and $V_1+\dots+V_d$,
$1\le i\le d$, is by Lukacs' theorem a characteristic property of the gamma
distribution; see, e.g., \citet[Section 2.6.1]{ngtita2011} for details.

The Dirichlet model for bivariate extreme value df was investigated by
\citet[Section 4.3]{colta91}, \citet[Example 3.6]{segers12a} studies the
Dirichlet model in arbitrary dimension. \citet[Appendix A]{bolda07} show that
each $D$-norm can be approximated by a $D$-norm generated by a mixture of
Dirichlet distributions.

The symmetric Dirichlet distribution is also an appealing parametric model
for a  rv that follows a \emph{generalized Pareto distribution} (GPD). Let
$U$ be uniformly distributed on $(0,1)$ and independent of the generator
$\bfZ$ as defined in \eqref{eqn:definition_generator_Dirichlet_distribution}.
Then
\[
\bfY:=-U\frac \bfone{\bfZ}
\]
follows a GPD with
\[
P(\bfY\le \bfx)=1-\frac 1\alpha E\left(\max_{1\le i\le d}\left(\abs{x_i}V_i\right)\right)
\]
for all $\bfx\le \bfzero\in\R^d$ with $\norm{\bfx}_\infty\le 1/d$. Equally,
\[
P(\bfY>\bfx)=\frac 1\alpha E\left(\min_{1\le i\le d}\left(\abs{x_i}V_i\right)\right).
\]

In the particular case $\alpha=1$ we obtain from the min-stability of the
exponential distribution on $[0,\infty)$
\[
E\left(\min_{1\le i\le d}V_i\right)=\frac 1d
\]
and, thus,
\[
P(\bfY>-c\bfone)=\frac cd,\qquad 0\le c\le 1/d.
\]
For an account of multivariate GPD we refer to \citet[Chapter 5]{fahure10}.

We discuss in what follows the generator constant function
\begin{equation*}
m(\alpha):=\norm{\bfone}_{D(\alpha)},\qquad \alpha > 0,
\end{equation*}
pertaining to the Dirichlet $D$-norms. We start with the bivariate case. From
the arguments in \citet[Section 4.3]{colta91} we obtain the representation
\[
\norm{(x,y)}_{D(\alpha)}=\abs{x} B\left(\alpha,\alpha+1, \frac{\abs{x}}{\abs{x}+\abs{y}}\right) + \abs{y} B\left(\alpha,\alpha+1, \frac{\abs{y}}{\abs{x}+\abs{y}}\right),
\]
where
\[
B(a,b,x)=\frac{\Gamma(a+b)}{\Gamma(a)\Gamma(b)} \int_0^x u^{a-1}(1-u)^{b-1}\,du,\qquad x\in[0,1],
\]
denotes the normalized incomplete beta function. The next result follows from
tedious but elementary computations.

\begin{prop}[The bivariate case]\label{dirichlet_generator_constant}
We have for all $\alpha>0$
\begin{equation*}
m(\alpha)=1+\frac{\Gamma\left(\alpha+\frac12\right)}{\sqrt{\pi}~\Gamma(\alpha+1)}=1+\frac{1}{\alpha
B\left(\alpha,\frac12\right)}
\end{equation*}
where $B(a,b) = \int_0^1 u^{a-1} (1-u)^{b-1} \, du$ denotes the beta
function.
\end{prop}

The fact that the function $m(\alpha)$ is decreasing and that it attains each
value in the interval $(1,d)$ is shown for arbitrary dimension in what
follows. Therefore we denote by $F_\alpha$ the df of the gamma distribution
with parameter $\alpha
> 0$, i.e.,
\begin{equation*}
F_{\alpha}(x)=\frac{\gamma(\alpha,x)}{\Gamma(\alpha)}=1-\frac{\Gamma(\alpha,x)}{\Gamma(\alpha)},\quad x>0,
\end{equation*}
where $\gamma(\alpha,x)=\int_0^x t^{\alpha-1}\exp(-t)\,dt$ and
$\Gamma(\alpha,x)=\int_x^{\infty}
t^{\alpha-1}\exp(-t)\,dt=\Gamma(\alpha)-\gamma(\alpha,x)$ are the lower and
the upper incomplete gamma function.

\begin{lemma}[Arbitrary dimension; {\citealp[Section 4.3]{colta91}}] \label{lem:the_limit_of_the_generator_constant}
Let $m(\alpha)=\norm{\bm 1}_{D(\alpha)}$ be the generator constant of the
$d$-dimensional Dirichlet generator. Then we have
$\lim_{\alpha\to0}m(\alpha)=d$ and $\lim_{\alpha\to\infty}m(\alpha)=1$.
\end{lemma}

The following auxiliary result will be the crucial tool in the proof of the
monotonicity of the Dirichlet-$D$-norm $\norm\cdot_{D(\alpha)}$ with respect
to the parameter $\alpha>0$, see below. It might be of interest of its own.

\begin{lemma}\label{lem:auxiliary_result_monotonicity}
Let $V_{ij}$, $1\le i\le d$, $1\le j\le n$, $d\in\N$, $n\ge 2$, be an array
of iid integrable rv. Then we have for arbitrary numbers $x_1,\dots,x_d\in\R$
\[
E\left(\max_{1\le i\le d}\left(x_i\frac{\sum_{j=1}^nV_{ij}}n\right)\right)\le  E\left(\max_{1\le i\le d}\left(x_i\frac{\sum_{j=1}^{n-1}V_{ij}}{n-1}\right)\right).
\]
\end{lemma}

\begin{proof}
The case $n=2$ is obvious: We have
\[
\max_{1\le i\le d}\left(x_i (V_{i1}+V_{i2})\right)\le \max_{1\le i\le d}\left(x_i V_{i1}\right) + \max_{1\le i\le d}\left(x_i V_{i2}\right)
\]
and, thus, by the identical distribution of $V_{i1}, V_{i2}$, $1\le i\le d$,
\[
E\left(\max_{1\le i\le d}\left(x_i (V_{i1}+V_{i2})\right)\right) \le 2 E\left(\max_{1\le i\le d}\left(x_i V_{i1}\right)\right).
\]

The case $n=3$ provides the crucial argument for a general $n$. Set
\[
x_{i^*}\left(V_{i^*1}+V_{i^*2}+V_{i^*3}\right) = \max_{1\le i\le d}\left(x_i\left(V_{i1}+V_{i2}+V_{i3}\right)\right).
\]
We have the obvious inequalities
\begin{align*}
x_{i^*}\left(V_{i^*1}+V_{i^*2}\right) &\le \max_{1\le i\le d}\left(x_i\left(V_{i1}+V_{i2}\right)\right),\\
x_{i^*}\left(V_{i^*1}+V_{i^*3}\right) &\le \max_{1\le i\le d}\left(x_i\left(V_{i1}+V_{i3}\right)\right),\\
x_{i^*}\left(V_{i^*2}+V_{i^*3}\right) &\le \max_{1\le i\le d}\left(x_i\left(V_{i2}+V_{i3}\right)\right).
\end{align*}
Summing up these inequalities we obtain
\begin{align*}
&2x_{i^*}\left(V_{i^*1}+V_{i^*2}+V_{i^*3}\right)\\
&\le \max_{1\le i\le d}\left(x_i\left(V_{i1}+V_{i2}\right)\right) + \max_{1\le i\le d}\left(x_i\left(V_{i1}+V_{i3}\right)\right) + \max_{1\le i\le d}\left(x_i\left(V_{i2}+V_{i3}\right)\right).
\end{align*}
Taking expectations on both sides yields
\[
E\left(\max_{1\le i\le d}\left(x_i\left(V_{i1}+V_{i2}+V_{i3}\right)\right)\right) \le\frac 32 E\left(\max_{1\le i\le d}\left(x_i\left(V_{i1}+V_{i2}\right)\right)\right),
\]
which proves the assertion for $n=3$. Repeating the preceding arguments
provides the assertion for a general $n$: Set
\[
x_{i^*}\sum_{j=1}^n V_{i^*j} =\max_{1\le i\le d} \left(x_i\sum_{j=1}^n V_{ij}\right).
\]
We have for all subsets $T\subset\set{1,\dots,n}$ with $n-1$ elements, i.e.,
$\abs T=n-1$,
\[
x_{i^*}\sum_{j\in T} V_{i^*j} \le \max_{1\le i\le d} \left(x_i\sum_{j\in T} V_{ij}\right).
\]
Summing up these $n$ inequalities we obtain
\[
(n-1)\, x_{i^*}\sum_{j=1}^n V_{i^*j} \le \sum_{T\subset\set{1,\dots,n},\abs T=n-1}  \max_{1\le i\le d} \left(x_i\sum_{j\in T} V_{ij}\right).
\]
Taking expectations on both sides now yields the assertion:
\[
E\left(\max_{1\le i\le d}\left(x_i\sum_{j=1}^n V_{ij}\right)\right) \le \frac n{n-1} E\left(\max_{1\le i\le d}\left(x_i\sum_{j=1}^{n-1} V_{ij}\right)\right).
\]
\end{proof}

The preceding result and the convolution theorem of the gamma distribution
provide the following bounds of the Dirichlet $D$-norm and the monotonicity
in $\alpha$.

\begin{prop}[Arbitrary dimension] \label{prop:monotonicity_bounds_of_Dirichlet-norm}
The Dirichlet $D$-norm $\norm\cdot_{D(\alpha)}$ is decreasing in $\alpha>0$,
i.e., we have for arbitrary $\bfx\in\R^d$
\[
\norm{\bfx}_{D(\alpha_1)}\ge \norm{\bfx}_{D(\alpha_2)},\qquad 0<\alpha_1\le \alpha_2.
\]
Moreover $\bfx\in\R^d$ and $0<\alpha_1<\alpha_2$ imply
\begin{equation*} 
  \alpha_1\norm\bfx_{D(\alpha_1)}
  \le \alpha_2\norm\bfx_{D(\alpha_2)}
  \le \alpha_1\norm\bfx_{D(\alpha_1)} + (\alpha_2-\alpha_1)\norm\bfx_{D(\alpha_2-\alpha_1)}.
\end{equation*}
\end{prop}

\begin{proof}
Choose $\bfx\in\R^d$ and put $g(\alpha):=\norm{\bfx}_{D(\alpha)}$,
$\alpha>0$. Note that the function $g$ is continuous. Suppose that there
exist $0<\alpha_1<\alpha_2$ with $g(\alpha_1)<g(\alpha_2)$. By the continuity
of $g$ we can find $\varepsilon>0$ and $k,n\in\N$, $k<n$, such that
$g(\varepsilon k)<g(\varepsilon n)$. Let $V_{ij}$, $1\le i\le d$, $1\le j\le
n$, be an array of independent and identically gamma distributed rv with
parameter $\varepsilon >0$. The convolution theorem of the gamma distribution
now implies
\[
g(\varepsilon k)= E\left(\max_{1\le i\le d}\left(\abs{x_i}\frac{\sum_{j=1}^k V_{ij}}{\varepsilon k}\right)\right) < g(\varepsilon n) =  E\left(\max_{1\le i\le d}\left(\abs{x_i}\frac{\sum_{j=1}^n V_{ij}}{\varepsilon n}\right)\right),
\]
which contradicts Lemma \ref{lem:auxiliary_result_monotonicity}.

Consider $0<\alpha_1<\alpha_2$ and let $V_1,\dots,V_d,W_1,\dots,W_d$ be
independent rv such that $V_i$ is gamma distributed with parameter $\alpha_1$
and $W_i$ is gamma distributed with parameter $\alpha_2-\alpha_1$, $1\le i\le
d$. The second assertion follows from
\eqref{eqn:representation_of_Dirichlet_D-norm}, the relation
\begin{align*}
  E\left( \max_{1\le i\le d} (\abs{x_i} V_i) \right)
  &\le E\left( \max_{1\le i\le d} (\abs{x_i} (V_i+W_i)) \right) \\
  &\le E\left( \max_{1\le i\le d} (\abs{x_i} V_i) + \max_{1\le i\le d} (\abs{x_i} W_i) \right)
\end{align*}
and the convolution theorem of the gamma distribution.
\end{proof}

\begin{lemma}\label{lem:generator_constant_of_exponential_gamma}
Let $V_1\dots,V_d$ be iid standard exponential distributed rv. The Dirichlet
$D$-norm $\norm\cdot_{D(1)}$ on $\R^d$ with generator
\[
\bfZ=d\left(\frac{V_i}{V_1+\dots+V_d}\right)_{i=1}^d
\]
has generator constant
\[
\norm{\bfone}_{D(1)}=\sum_{k=1}^d\frac 1k.
\]
\end{lemma}

The generator constant of a general \emph{bivariate} Dirichlet $D$-norm was
computed in Lemma \ref{dirichlet_generator_constant}. To the best of our
knowledge, the preceding result, with $\alpha=1$, provides the only exact
computation of $\norm{\bfone}_{D(\alpha)}$ for arbitrary dimension.  Some
representation is given by \citet{nada08}.

\begin{proof}
The following argument is taken from \citet[Section 33.3]{balabasu1996}.
Using the memoryless property of the exponential distribution one can
generate order statistics $V_{1:d}\le \dots \le V_{d:d}$ from the standard
exponential distribution as follows:
\begin{itemize}
  \item Generate $d$ independent standard exponential distributed rv $V_1,\dots,V_d$.
  \item Then
      \[
      V_{i:d}:= V_{i-1:d}+ \frac{V_i}{d-i+1},\qquad i=1,\dots,d,
      \]
      with $V_{0:d}=0$ are the required order statistics.
\end{itemize}
Hence we obtain
\[
\norm{\bfone}_{D(1)}=E(V_{d:d})=\sum_{i=1}^d\frac{E(V_i)}{d-i+1}=\sum_{i=1}^d \frac 1i.
\]
\end{proof}

The fact that the function $m(\alpha)$ is continuous and decreasing with
$\lim_{\alpha\downarrow 0}m(\alpha)=d$, $\lim_{\alpha\uparrow
\infty}m(\alpha)=1$ shows that the family of symmetric Dirichlet
distributions is a parametric family of generators of $D$-norms in arbitrary
dimension, which attains each value between independence
($\norm{\bfone}_D=d$) and complete dependence ($\norm{\bfone}_D=1$). The
generator of the symmetric Dirichlet distribution is well-known and easy to
simulate. This makes the family of symmetric Dirichlet distributions quite an
attractive parametric model of $D$-norms.

\end{document}